\documentclass[10pt,a4paper]{article}

\usepackage[utf8]{inputenc}
\usepackage[T1]{fontenc}
\usepackage{float}
\usepackage{booktabs}
\usepackage[english]{babel} 
\usepackage{amsmath, amssymb, amsthm, amsfonts}
\usepackage{mathrsfs} % Para fuentes bonitas como \mathsc
\usepackage{geometry}
\usepackage{graphicx}
\usepackage{hyperref}
\usepackage{xcolor}
\usepackage{enumitem}

\newtheorem{theorem}{Theorem}[section]

\newtheorem{proposition}[theorem]{Proposition}

\theoremstyle{definition}
\newtheorem{definition}[theorem]{Definition}

\theoremstyle{remark}
\newtheorem{remark}[theorem]{Remark}

\title{\textbf{A Unified Spectral Framework for Aging, Heterogeneous, and Distributed Order Systems via Weighted Weyl-Sonine Operators}}

\author{
    \textbf{Gustavo Dorrego} \\
    \small{Department of Mathematics, Facultad de Ciencias Exactas y Naturales y Agrimensura}\\
    \small{Universidad Nacional del Nordeste, Corrientes, Argentina}\\
    \small{\texttt{gadorrego@exa.unne.edu.ar}}
}

\date{\today}

\begin{document}

\maketitle

\begin{abstract}
While General Fractional Calculus has successfully expanded the scope of memory operators beyond power-laws, standard formulations remain predominantly restricted to the half-line via Riemann-Liouville or Caputo definitions. This constraint artificially truncates the system's history, limiting the thermodynamic consistency required for modeling processes on unbounded domains. To overcome these barriers, we construct the \textbf{Weighted Weyl-Sonine Framework}, a generalized formalism that extends non-local theory to the entire real line without history truncation.

Unlike recent algebraic approaches based on conjugation for finite intervals, we develop a rigorous harmonic analysis framework. Our central contribution is the \textbf{Generalized Spectral Mapping Theorem}, which establishes the Weighted Fourier Transform as a unitary diagonalization map for these operators. This result allows us to rigorously classify and solve distinct physical regimes under a single algebraic structure. We explicitly derive exact solutions for \textit{diffusive relaxation} (governed by Complete Bernstein Functions), \textit{inertial wave propagation} (exhibiting oscillatory dynamics), and \textit{retarded aging} (via distributed order), proving that our framework unifies the description of anomalous transport and wave mechanics in complex, time-deformed media.
\end{abstract}\vspace{0.5cm}
    \noindent 
    
    \noindent
    

\section{Introduction}
\label{sec:intro}

The framework of Fractional Calculus has firmly established itself as a fundamental tool for modeling anomalous dynamics in complex systems. Traditionally, operators based on the Riemann-Liouville or Caputo definitions have been the gold standard for describing processes with power-law memory. These operators have successfully captured phenomena in linear viscoelasticity \cite{Bagley1983, Mainardi2010}, anomalous transport and diffusion \cite{Metzler2000}, and dielectric relaxation processes \cite{Hilfer2000, Cole1941}. The ubiquity of the power-law kernel $k(t) \propto t^{-\alpha}$ arises from its scale-invariance, which naturally models the self-similar fractal structures inherent to these media \cite{West2003, Nigmatullin1986}.

However, experimental evidence in highly heterogeneous materials—such as biological tissues, amorphous polymers, and porous media—suggests that memory effects are often far more intricate. In these systems, relaxation rates may evolve over time due to structural changes (``aging''), or the memory function may exhibit transitions between different decay regimes. For instance, in aging polymers or hardening concrete, the system's response depends explicitly on its age, violating the time-translation invariance inherent to classical convolution operators \cite{Sun2019, Sokolov2012}. Furthermore, complex porous media often exhibit ultra-slow diffusion best described by distributed orders \cite{Meerschaert2006}, while inertial viscoelasticity can lead to oscillating relaxation dynamics that monotonic power-laws fail to capture \cite{Mainardi2014}.

To address these limitations, the theory of General Fractional Calculus (GFC) has emerged as a rigorous generalization, replacing the standard power-law with arbitrary kernels $k(t)$ satisfying the Sonine condition \cite{Luchko2021}. Yet, a fundamental question remains regarding the domain of these operators. As eloquently argued by Stinga \cite{Stinga}, true ``memory effects''—likened to an elephant never forgetting its complete history—are best captured by operators of the Weyl-Marchaud type defined on the entire real line $\mathbb{R}$. Restricting analysis to $t>0$ artificially truncates the system's history, whereas a complete thermodynamic description requires integrating over the infinite past $(-\infty, t)$.

Recently, significant strides have been made in understanding the structure of these generalized operators. A major breakthrough was achieved by Al-Refai and Fernandez \cite{Fernandez2023}, who provided a unified algebraic framework for Sonine operators via the method of conjugations. Their work elegantly demonstrates that many weighted operators on the half-line $[0, \infty)$ can be constructed by conjugating standard derivatives with invertible linear operators, effectively unifying a wide range of fractional definitions under a common algebraic umbrella. However, their approach is primarily tailored to initial value problems using the Laplace transform on semi-axes. Consequently, the \textit{spectral analysis} of ``full history'' operators on the unbounded domain $\mathbb{R}$—particularly those combining general Sonine kernels with time-dependent weights via the Fourier transform—remains an open problem. While abstract approaches to evolution equations have recently appeared (e.g., Lizama and Murillo-Arcila, 2025 \cite{Lizama2025}), an explicit harmonic analysis framework allowing for concrete calculations in these heterogeneous settings is still lacking.

This gap has led to a fragmentation of the field into a ``zoo'' of disconnected operators. Researchers must often choose between the power-law memory of Riemann-Liouville, the logarithmic scaling of Hadamard, the non-singular kernels of Caputo-Fabrizio, or tempered behaviors. While the Generalized Fractional Calculus provides the algebraic structure for arbitrary kernels, the physical modeling of aging viscoelasticity, inertial oscillation, and ultra-slow diffusion demands a framework that simultaneously handles general memory geometry and time-domain heterogeneity.

In a recent study \cite{Dorrego2026}, we laid the foundations for a unified approach by constructing weighted Weyl derivatives for the specific case of power-law kernels. In this work, we significantly extend those results by constructing the \textbf{Weighted Weyl-Sonine Operators}. This class acts as a ``master framework'' that generalizes operators of differentiation and integration by incorporating three independent degrees of freedom:
\begin{itemize}
    \item The \textbf{Sonine kernel} $k(z)$, which determines the relaxation spectrum (e.g., monotonic vs. oscillatory).
    \item The \textbf{scale function} $\psi(t)$, which determines the domain geometry (e.g., invariance of scale vs. translation).
    \item The \textbf{weight function} $\omega(t)$, which controls aging or tempering effects.
\end{itemize}

Our main contribution is a \textbf{Generalized Spectral Mapping Theorem}, which proves that the Weighted Fourier Transform acts as a universal diagonalization map for these complex operators. This result is non-trivial, as it requires handling the interplay between the Sonine convolution condition and non-translation-invariant weight functions.

The paper is organized as follows: Section \ref{sec:preliminaries} reviews the preliminaries. Section \ref{sec:weyl_sonine_framework} constructs the operators and proves the Spectral Mapping Theorem. Finally, Section 4 applies this theory to two distinct cases: a generalized relaxation equation with oscillating memory (Bessel kernel) and the problem of distributed-order evolution in aging media, deriving explicit solutions via the spectral method.
\section{Preliminaries and Functional Framework}
\label{sec:preliminaries}

In this section, we summarize the necessary functional spaces and integral transform tools developed in \cite{DorregoPaper1}, and we recall the fundamental properties of Sonine kernels \cite{Luchko2021, Luchko2021+}.

\subsection{General Fractional Calculus with Sonine Kernels}
\label{subsec:luchko_prelims}

Following the framework established by Luchko \cite{Luchko2021, Luchko2021+}, we first review the construction of general fractional operators on the semi-axis $\mathbb{R}_+$. This serves as the local foundation for our non-local Weyl operators.

% --- DEFINICIÓN DE ESPACIOS ---

\begin{definition}[Space $C_{-1}$]
\label{def:space_C_minus1}
Let $C_{-1}(0, \infty)$ be the space of functions $f$ defined on $(0, \infty)$ such that $f(t) = t^p g(t)$ for some real number $p > -1$ and $g \in C[0, \infty)$. This space contains both functions continuous at the origin (if $p \ge 0$) and functions with integrable singularities (if $-1 < p < 0$).
\end{definition}

% --- DEFINICIÓN DE OPERADORES Y CONDICIÓN DE SONINE ---

The core of this theory relies on the so-called \textit{Sonine condition}, which generalizes the relation between the power-law kernels of the Riemann-Liouville calculus.

\begin{definition}[The Sonine Condition]
    \label{def:sonine_condition}
    Let $k, \kappa \in C_{-1}(0, \infty)$ be two real-valued kernels. The pair $(k, \kappa)$ is said to satisfy the Sonine condition if their Dirichlet convolution is identically unity:
    \begin{equation}
        (k * \kappa)(t) = \int_0^t k(t-\tau)\kappa(\tau) \, d\tau = 1, \quad \forall t > 0.
    \end{equation}
    The set of all such pairs $\mathcal{K}$ is called the set of \textit{Sonine kernels}.
\end{definition}

Based on these kernels, Luchko defines the General Fractional Integral (GFI) and Derivative (GFD) which unify previous fractional operators.

\begin{definition}[General Fractional Operators]
    For a kernel $k$ satisfying the Sonine condition, the General Fractional Integral denoted by $\mathcal{I}_{(k)}$ is defined as the convolution:
    \begin{equation}
        \mathcal{I}_{(k)} f(t) := (k * f)(t) = \int_0^t k(t-\tau) f(\tau) \, d\tau, \quad t > 0.
    \end{equation}
    The associated General Fractional Derivative (GFD) of regularized type, denoted by $\mathbb{D}_{(\kappa)}$, is defined using the associated kernel $\kappa$:
    \begin{equation}
        \mathbb{D}_{(\kappa)} f(t) := \frac{d}{dt} \int_0^t \kappa(t-\tau) f(\tau) \, d\tau - \kappa(t)f(0).
        \label{eq:GFD_Luchko}
    \end{equation}
\end{definition}

% --- EXTENSIÓN A ÓRDENES ARBITRARIOS ---

\begin{definition}[Generalized Sonine Pairs of Arbitrary Order]
\label{def:gen_sonine_luchko}
Following the framework introduced by Luchko \cite{Luchko2021+}, let $n \in \mathbb{N}$ be a natural number. We say that a kernel $\kappa$ belongs to the set of generalized Sonine kernels $\mathcal{L}_n$ if there exists an associated kernel $k$ such that the pair $(\kappa, k)$ satisfies the generalized convolution condition:
\begin{equation}
    (\kappa * k)(t) = \{1\}^{*n}(t) = \frac{t^{n-1}}{(n-1)!}, \quad \forall t > 0,
\end{equation}
where $\{1\}^{*n}(t)$ denotes the $n$-fold convolution of the unit function.
\end{definition}

\begin{remark}
For $n=1$, this recovers the classical Sonine condition. For $n > 1$, this definition allows for the construction of General Fractional Derivatives of arbitrary order, defined by Luchko as:
\begin{equation}
    \mathbb{D}_{(k)} f(t) = \frac{d^n}{dt^n} (k * f)(t) - \sum_{j=0}^{n-1} f^{(j)}(0) \frac{d^n}{dt^n} (k * \{1\}^{*(j+1)})(t).
\end{equation}
This formulation is essential for modeling high-order distributed systems where the memory kernel does not necessarily behave as a singular power law near the origin.
\end{remark}

\begin{remark}[Transition to Weighted Weyl Operators]
    In our work, we extend these concepts to the \textit{entire real line} (Weyl type) and introduce time-dependent weights. While Luchko's operators $\mathbb{D}_{(\kappa)}$ are defined for $t>0$ and involve initial conditions ($f(0)$), our Weighted Weyl operators act on the complete history $t \in (-\infty, \infty)$ and naturally vanish at $-\infty$ for functions in the Schwartz space, removing the need for initial value terms while preserving the algebraic structure of the Sonine pairs.
\end{remark}

\subsection{Weighted Spaces and Conjugation}

Let $\omega \in C^\infty(\mathbb{R})$ be a strictly positive weight function and let $\psi \in C^\infty(\mathbb{R})$ be a time-scale function satisfying the following conditions:
\begin{enumerate}
    \item \textbf{Surjectivity:} $\psi(\mathbb{R}) = \mathbb{R}$ (covering the full timeline).
    \item \textbf{Non-degeneracy:} $\psi'(t) > 0$ for all $t \in \mathbb{R}$.
\end{enumerate}

The condition $\psi(\mathbb{R}) = \mathbb{R}$ ensures that the induced spectral domain covers the entire frequency line $\xi \in \mathbb{R}$, while $\psi'(t) > 0$ guarantees that the change of variables is a global diffeomorphism, preventing singularities in the Jacobian.

We consider the weighted $L^2$-space denoted by $L^2_{\psi,\omega}(\mathbb{R})$, equipped with the inner product:
\begin{equation}
    \langle u, v \rangle_{\psi,\omega} = \int_{-\infty}^{\infty} u(t) \overline{v(t)} \omega(t) \psi'(t) \, dt.
\end{equation}

To handle the heterogeneity of the medium, we utilize the conjugation map approach proposed by Fernandez and Fahad \cite{Fernandez2022}. We define the isometric isomorphism $\mathcal{T}_{\psi,\omega}: L^2_{\psi,\omega}(\mathbb{R}) \to L^2(\mathbb{R})$ as:
\begin{equation}
    \label{eq:isometry_T}
    (\mathcal{T}_{\psi,\omega} u)(\xi) = \left( u \circ \psi^{-1} \right)(\xi) \cdot \left( \omega \circ \psi^{-1} \right)(\xi).
\end{equation}
This operator acts as a "rectifier," mapping functions from the deformed, heterogeneous time-space to the homogeneous standard space.

\subsection{The Weighted Fourier Transform}

To diagonalize operators with time-dependent coefficients, we employ the Weighted Fourier Transform constructed in our previous work \cite{DorregoWFT}. While the general theory was developed for $\mathbb{R}^n$, here we present the one-dimensional formulation adapted to the global time domain $\mathbb{R}$.

We introduce the weighted Lebesgue spaces, denoted by $L^p_{\psi,\omega}(\mathbb{R})$, defined as the completion of $C_c^\infty(\mathbb{R})$ with respect to the norm associated with the weighted measure $d\mu_{\psi,\omega}(t) = \omega(t)\psi'(t)dt$:
\begin{equation}
    \|u\|_{L^p_{\psi,\omega}} := \left( \int_{\mathbb{R}} |u(t)|^p \omega(t) \psi'(t) \, dt \right)^{1/p} < \infty.
\end{equation}

\begin{definition} \label{def:weighted_FT_1D}
For a function $f \in L_{\psi,\omega}^{1}(\mathbb{R})$, the Weighted Fourier Transform with respect to the scale $\psi$ and weight $\omega$, denoted by $\mathcal{F}_{\psi,\omega}$, is defined as:
\begin{equation} \label{eq:weighted_FT_formula_1D}
    [\mathcal{F}_{\psi,\omega}f](\xi) = \frac{1}{\sqrt{2\pi}} \int_{-\infty}^{\infty} e^{-i\xi \psi(t)} \omega(t) f(t) \psi'(t) \, dt.
\end{equation}
\end{definition}

Analytically, this operator acts as a conjugation of the classical Fourier transform $\mathcal{F}$ by structural isomorphisms. Specifically, if we define the composition operator $Q_{\psi} g(t) = g(\psi(t))$ and the multiplication operator $M_{\omega} f(t) = \omega(t)f(t)$, the definition \eqref{eq:weighted_FT_formula_1D} admits the factorization:
\begin{equation}
    \mathcal{F}_{\psi,\omega} = \mathcal{F} \circ Q_{\psi}^{-1} \circ M_{\omega}.
\end{equation}
This factorization is crucial, as it transfers the unitarity and inversion properties of the classical transform directly to the weighted setting via the induced isometry. The inverse transform allows for the recovery of the time-domain solution.

\begin{theorem}[Inversion Formula]
Let $f$ be a function such that both $f$ and $\mathcal{F}_{\psi,\omega}f$ are integrable in the appropriate weighted sense. Then, the inverse transform is given pointwise by:
\begin{equation}
    f(t) = \frac{1}{\omega(t)\sqrt{2\pi}} \int_{-\infty}^{\infty} e^{i \xi \psi(t)} [\mathcal{F}_{\psi,\omega} f](\xi) \, d\xi.
\end{equation}
\end{theorem}

\begin{proof}
The proof follows directly from the $n$-dimensional case presented in \cite{DorregoWFT} by setting $n=1$, $\phi(x) = \psi(t)$, and noting that the Jacobian determinant $|J_{\phi}(x)|$ reduces to $\psi'(t)$ since $\psi$ is strictly increasing. By the change of variables $u = \psi(t)$, the integral reduces to the classical Fourier inversion theorem.
\end{proof}

\begin{remark}
A fundamental property of $\mathcal{F}_{\psi,\omega}$ is that it extends to a unitary operator (isometry) on $L^2_{\psi, \omega}(\mathbb{R})$ under suitable conditions on $\omega$. This generalized Plancherel identity (see \cite{DorregoWFT}) is essential for the energy estimates derived in Section 4.
\end{remark}

\begin{definition}[Weighted Schwartz Space as a Fréchet Space]
\label{def:weighted_schwartz}
Let $\psi: \mathbb{R} \to \mathbb{R}$ be a smooth diffeomorphism and let $\omega \in C^\infty(\mathbb{R})$ be a non-vanishing weight. The \textit{Weighted Schwartz Space}, denoted by $\mathcal{S}_{\psi, \omega}(\mathbb{R})$, is defined as the isomorphic image of the standard Schwartz space $\mathcal{S}(\mathbb{R})$ under the mapping $\mathcal{T}_{\psi, \omega} = M_\omega^{-1} \circ Q_\psi$:
\begin{equation}
    \mathcal{S}_{\psi, \omega}(\mathbb{R}) := \left\{ u : \mathbb{R} \to \mathbb{C} \mid \exists f \in \mathcal{S}(\mathbb{R}) \text{ such that } u(t) = \frac{f(\psi(t))}{\omega(t)} \right\}.
\end{equation}
We equip $\mathcal{S}_{\psi, \omega}(\mathbb{R})$ with the locally convex topology induced by this bijection. Specifically, convergence $u_n \to u$ in $\mathcal{S}_{\psi, \omega}$ corresponds to the convergence of the pull-backs $f_n = \omega \cdot (u_n \circ \psi^{-1})$ in the standard Fréchet topology of $\mathcal{S}(\mathbb{R})$.
\end{definition}

\begin{proposition}[Density and Boundary Decay]
\label{prop:density_boundary}
The space $\mathcal{S}_{\psi, \omega}(\mathbb{R})$ is dense in the weighted Hilbert space $L^2_{\psi, \omega}(\mathbb{R})$. Furthermore, the rapid decay inherited from $\mathcal{S}(\mathbb{R})$ ensures that for any $u \in \mathcal{S}_{\psi, \omega}(\mathbb{R})$, both the function and its weighted derivatives vanish at infinity:
\begin{equation}
    \lim_{t \to \pm \infty} u(t) = 0 \quad \text{and} \quad \lim_{t \to \pm \infty} \mathfrak{D}_{\psi, \omega}^{n} u(t) = 0 \quad \forall n \in \mathbb{N},
\end{equation}
where $\mathfrak{D}_{\psi, \omega}^{n}$ denotes the iterated weighted differential operator defined by:
\begin{equation}\label{eq:nth_order_def}
    \mathfrak{D}_{\psi, \omega}^{n} u(t) := \frac{1}{\omega(t)} \left( \frac{1}{\psi'(t)} \frac{d}{dt} \right)^n \big( \omega(t) u(t) \big).
\end{equation}
\end{proposition}

\begin{proof}
The proof relies on the unitary equivalence established by the mapping $\mathcal{T}_{\psi, \omega}$.

\textbf{1. Density:}
Recall that the standard Schwartz space $\mathcal{S}(\mathbb{R})$ is dense in $L^2(\mathbb{R})$. Since the operator $\mathcal{T}_{\psi, \omega} = M_\omega^{-1} Q_\psi$ is a unitary isomorphism from $L^2(\mathbb{R})$ to $L^2_{\psi, \omega}(\mathbb{R})$ (preserving the inner product structure), it maps dense sets to dense sets. Therefore, the image $\mathcal{S}_{\psi, \omega}(\mathbb{R}) = \mathcal{T}_{\psi, \omega}(\mathcal{S}(\mathbb{R}))$ is necessarily dense in $L^2_{\psi, \omega}(\mathbb{R})$.

\textbf{2. Boundary Decay:}
Let $u \in \mathcal{S}_{\psi, \omega}(\mathbb{R})$. By definition, there exists a unique $f \in \mathcal{S}(\mathbb{R})$ such that $u(t) = f(\psi(t))/\omega(t)$, or equivalently $\omega(t)u(t) = f(\psi(t))$.
Substituting this into the definition of the weighted derivative \eqref{eq:nth_order_def}, we apply the chain rule:
$$
\mathfrak{D}_{\psi, \omega}^{n} u(t) = \frac{1}{\omega(t)} \left( \frac{1}{\psi'(t)} \frac{d}{dt} \right)^n f(\psi(t)).
$$
Let $y = \psi(t)$. Since $\frac{1}{\psi'(t)} \frac{d}{dt} = \frac{d}{dy}$, the operator simplifies strictly to the classical derivative in the transformed domain:
$$
\left( \frac{1}{\psi'(t)} \frac{d}{dt} \right)^n f(\psi(t)) = f^{(n)}(y) \Big|_{y=\psi(t)}.
$$
Thus, we have the relation:
$$
\omega(t) \cdot \mathfrak{D}_{\psi, \omega}^{n} u(t) = f^{(n)}(\psi(t)).
$$
Since $f \in \mathcal{S}(\mathbb{R})$, its derivatives $f^{(n)}$ decay rapidly at infinity. Assuming the weight $\omega(t)$ does not grow superexponentially relative to the decay of $f$ (a condition satisfied by admissible weights), the term on the RHS vanishes as $\psi(t) \to \pm \infty$. Consequently, $\mathfrak{D}_{\psi, \omega}^{n} u(t) \to 0$ as $t \to \pm \infty$.
\end{proof}

%-----------------------------
\section{The Weighted Weyl-Sonine Framework and Spectral Analysis}
\label{sec:weyl_sonine_framework}

In this section, we construct the unified operator. To ensure the convergence of integrals over the history $(-\infty, t]$, we must adapt the classical Sonine conditions (typically defined on bounded intervals or $\mathbb{R}_+$) to the context of Weyl operators on the entire real line.

\subsection{Admissible Kernels and Operators}

Unlike the Riemann-Liouville case, where integration starts at a finite point $a$, the Weyl operator integrates the entire history. This requires the memory kernel to decay sufficiently fast as time delay increases.

\begin{definition}[Generalized Tempered Sonine-Weyl Pair of Order $n$]
\label{def:tempered_pair}
Let $n \in \mathbb{N}$. Let $k, \kappa \in \mathcal{S}'(\mathbb{R})$ be two tempered distributions. The pair $(k, \kappa)$ is called an \textit{admissible tempered pair of order $n$} if the following conditions are satisfied:

\begin{enumerate}
    \item \textbf{Support Condition:} Both distributions are supported on the closed semi-axis $[0, \infty)$. That is, $k, \kappa \in \mathcal{S}'_+(\mathbb{R})$.
    
    \item \textbf{Distributional Generalized Condition:} The convolution of the kernels yields the $n$-th primitive of the Dirac delta, which corresponds to the truncated power function distribution:
    \begin{equation} \label{eq:dist_convolution}
        k * \kappa = \frac{t_+^{n-1}}{(n-1)!} \quad \text{in } \mathcal{S}'(\mathbb{R}),
    \end{equation}
    where $t_+^{\lambda}$ is the standard truncated power distribution defined by $\langle t_+^{\lambda}, \phi \rangle = \int_0^\infty t^\lambda \phi(t) dt$.
    
    For the classical case $n=1$, this recovers the Heaviside step function: $k * \kappa = \theta(t)$.
    
    \item \textbf{Spectral Condition:} To ensure the operator corresponds to a fractional process, the Laplace symbol $\hat{k}(s)$ (defined for $\text{Re}(s)>0$) satisfies algebraic growth/decay conditions at infinity consistent with a fractional order strictly between $n-1$ and $n$.
\end{enumerate}
\end{definition}

\begin{remark}[Diffusive vs. Oscillatory Regimes]
\label{rem:diffusive_vs_oscillatory}
It is crucial to distinguish between two physical regimes covered by Definition \ref{def:tempered_pair}:
\begin{enumerate}
    \item \textbf{The Diffusive Regime:} Characterized by kernels $k(t)$ whose Laplace transforms are \textit{Complete Bernstein Functions} (CBF). These operators generate semi-groups that preserve positivity, consistent with thermodynamic relaxation.
    \item \textbf{The Oscillatory Regime:} Characterized by kernels (e.g., Bessel functions $J_0(t)$ associated with wave equations) that satisfy the general Sonine condition but \textbf{do not} belong to the CBF class. In this regime, the sign-preservation property is naturally lost, reflecting the inertial nature of the solution.
\end{enumerate}
The Spectral Mapping Theorems presented in this work rely on the algebraic convolution structure and hold for both regimes, whereas the maximum principles are restricted to the diffusive case.
\end{remark}

\begin{remark}[From Local Functions to Global Distributions]
\label{rem:local_to_global}
Note that any classical Sonine kernel $k_{loc}(t) \in C_{-1}(0,\infty)$ identified in Section 2.1 generates a regular tempered distribution $k \in \mathcal{S}'(\mathbb{R})$ by extending it as zero for $t<0$. Condition \eqref{eq:dist_convolution} generalizes the classical relation $\int_0^t k(t-\tau)\kappa(\tau)d\tau = 1$ (case $n=1$), allowing us to include kernels that involve singular distributions (like $\delta(t)$) which are essential for modeling elastic solids or purely viscous fluids, but are excluded in the classical function-space theory.
\end{remark}

\begin{remark}[Consistency with Classical and Generalized Kernels] 
\label{rem:consistency}
This definition unifies the classical and generalized frameworks within distribution theory:
\begin{itemize}
    \item \textbf{Case $n=1$ (Classical):} We recover the standard Sonine condition $k * \kappa = 1$ (for $t>0$). This covers standard models like subdiffusion, where kernels are singular at the origin.
    \item \textbf{Case $n>1$ (High-Order):} This allows for kernels that are continuous at the origin (regular kernels). In this setting, the operator $\mathfrak{D}_{\psi,\omega}^{(\kappa)}$ acts as a weighted derivative of order $n$, involving the $n$-th derivative of the convolution.
\end{itemize}
\end{remark}
% --- FIGURA 1: MEMORY KERNELS ---
% Ubicación: Justo después de definir la clase de Sonine
\begin{figure}[htbp]
    \centering
    % Asegúrate de que el nombre del archivo coincida con el tuyo
    \includegraphics[width=0.75\textwidth]{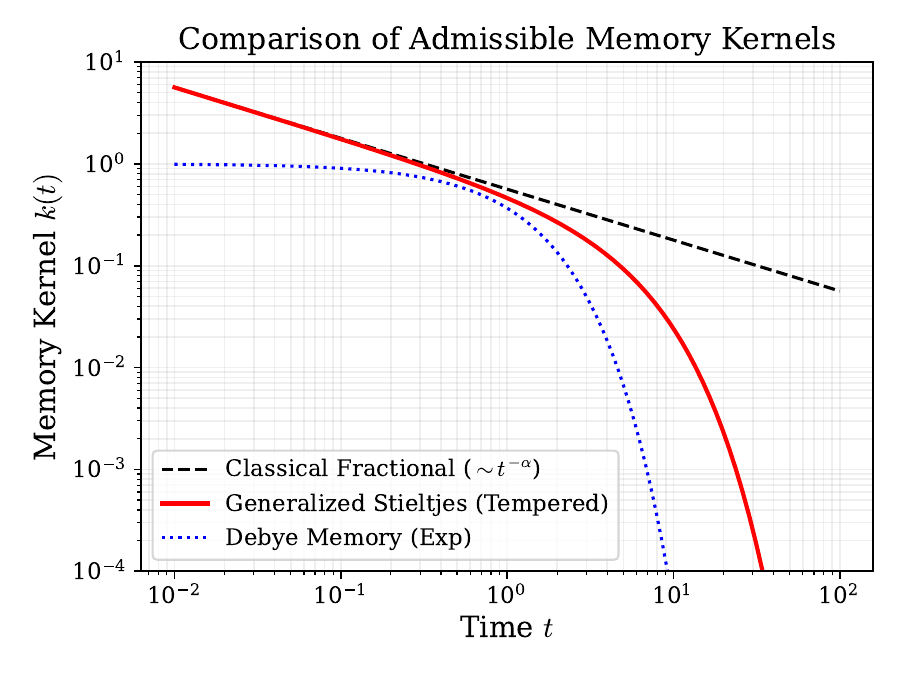} 
    \caption{\textbf{Admissible Memory Kernels.} Comparison of decay rates for different Sonine kernels. While classical fractional calculus relies on pure power-laws (blue), our framework incorporates tempered kernels (orange) and exponential decays (green), allowing for finite propagation speeds and thermodynamic consistency.}
    \label{fig:kernels}
\end{figure}
\subsection{Weighted Weyl-Sonine Operators.}
We now introduce the main operator of this work, constructed via the conjugation of the classical Sonine integral with the weight and scale operators defined in \cite{Fernandez2022}.

\begin{definition}[Auxiliary Structure Operators]
\label{def:structure_operators}
Let $(\psi, \omega)$ be an admissible structure pair. We define the following operators acting on measurable functions $f: \mathbb{R} \to \mathbb{C}$:

\begin{enumerate}
    \item \textbf{The Multiplication Operator} $M_{\omega}$, which weights the function:
    \begin{equation}
        (M_{\omega} f)(t) := \omega(t) f(t).
    \end{equation}
    
    \item \textbf{The Substitution (Time-Warping) Operator} $Q_{\psi}$, which deforms the domain:
    \begin{equation}
        (Q_{\psi} f)(t) := f(\psi(t)).
    \end{equation}
\end{enumerate}

Since $\psi$ is a diffeomorphism and $\omega$ is non-vanishing, both operators are invertible on suitable spaces. Specifically, the inverse substitution is given by $(Q_{\psi}^{-1} g)(t) = g(\psi^{-1}(t))$.
\end{definition}

\begin{remark}
These operators allow us to factorize the Weighted Fourier Transform as a composition: $\mathcal{F}_{\psi,\omega} = \mathcal{F} \circ Q_{\psi}^{-1} \circ M_{\omega}$. % 
This factorization is the key to the spectral analysis.
\end{remark}

\begin{definition}[Left-Sided Weighted Weyl-Sonine Integral]
\label{def:weighted_integral}
Let $(k, \kappa)$ be an admissible Sonine pair and let $u$ be a function in the weighted space $\mathcal{S}_{\psi,\omega}$. We define the Left-Sided Weighted Weyl-Sonine Integral, denoted by $\mathfrak{I}_{\psi,\omega}^{(k)}$, as follows:

\textbf{1. Explicit Form:}
The operator is defined pointwise by the weighted integral:
\begin{equation}
    \mathfrak{I}_{\psi,\omega}^{(k)} u(t) := \frac{1}{\omega(t)} \int_{-\infty}^t k\Big(\psi(t) - \psi(\tau)\Big) \omega(\tau) u(\tau) \psi'(\tau) \, d\tau.
\end{equation}

\textbf{2. Structural Form:}
In terms of the auxiliary operators from Definition \ref{def:structure_operators}, the integral acts as a conjugated convolution:
\begin{equation}
    \mathfrak{I}_{\psi,\omega}^{(k)} = M_{\omega}^{-1} Q_{\psi} (k *) Q_{\psi}^{-1} M_{\omega},
\end{equation}
where $(k *)$ denotes the standard convolution operator on $\mathbb{R}$.
\end{definition}

\noindent \textbf{Physical Interpretation:}
This operator models a process with memory in a heterogeneous medium:
\begin{itemize}
    \item The term $\Delta_\psi(t,\tau) = \psi(t) - \psi(\tau)$ represents the \textit{subjective time elapsed} between the past event $\tau$ and the present $t$.
    \item The kernel $k(\Delta_\psi)$ assigns a weight to this memory (typically singular at the origin).
    \item The ratio $\omega(\tau)/\omega(t)$ accounts for the \textit{aging} or changing density of the medium during the process.
\end{itemize}

\begin{definition}[Weighted Weyl-Sonine Derivative of Order $n$]
\label{def:derivada}
Let $n \in \mathbb{N}$ and let $(k, \kappa)$ be an admissible tempered Sonine-Weyl pair of order $n$. We define the Weighted Weyl-Sonine derivative, denoted by $\mathfrak{D}_{\psi,\omega}^{(\kappa)}$, as the regularized inverse of the integral operator.

\textbf{1. Explicit Form:} 
For a suitable function $u$, the operator is defined pointwise by applying the weighted differential operator $n$ times:
\begin{equation}\label{eq:sonine_def}
    \mathfrak{D}_{\psi,\omega}^{(\kappa)}u(t) = \frac{1}{\omega(t)} \left( \frac{1}{\psi^{\prime}(t)} \frac{d}{dt} \right)^n \int_{-\infty}^{t} \kappa(\psi(t)-\psi(\tau))\omega(\tau)u(\tau)\psi^{\prime}(\tau)d\tau.
\end{equation}

\textbf{2. Structural Form:} 
By identifying the differential operator $\frac{1}{\psi'(t)}\frac{d}{dt}$ with the derivative $\frac{d}{dy}$ in the warped time $y=\psi(t)$, the operator admits the following conjugation structure:
\begin{equation}\label{eq:sonine_abstract}
    \mathfrak{D}_{\psi,\omega}^{(\kappa)} = M_{\omega}^{-1} Q_{\psi} \left( \frac{d^n}{dy^n} \circ (\kappa *) \right) Q_{\psi}^{-1} M_{\omega}.
\end{equation}
This abstract form emphasizes that our operator is structurally equivalent to a classical convolution derivative of order $n$, conjugated by the geometry $\psi$ and weight $\omega$.
\end{definition}

Note that for $n=1$, the operator simplifies to the weighted first derivative $\frac{1}{\psi^{\prime}(t)} \frac{d}{dt}$, recovering the classical case. For $n \ge 2$, this formulation captures the higher-order dynamics required by the generalized Sonine pairs.

\begin{remark}[Algebraic Structure and the Global Equivalence of Forms]
\label{rem:structure_and_equivalence}
The construction in Definitions \ref{def:weighted_integral}-\ref{def:derivada} hereditarily preserves the convolution algebra of the standard Sonine calculus, albeit non-trivially deformed by the diffeomorphism $\psi$ and the weight $\omega$.

This global structure also resolves the distinction between Riemann-Liouville and Caputo definitions commonly found in General Fractional Derivatives on finite intervals $[0,T]$ (e.g., Luchko \cite{Luchko2021+}). In those local settings, the non-commutativity of differentiation and integration necessitates explicit regularization terms (initial values). In contrast, the rapid decay of test functions in our weighted Schwartz space $\mathcal{S}_{\psi,\omega}(\mathbb{R})$ ensures that all boundary terms vanish naturally:
$$
\lim_{t \to -\infty} \left( \frac{1}{\psi'(t)} \frac{d}{dt} \right)^k u(t) = 0, \quad \forall k < n.
$$
Consequently, the differential and integral operators commute strictly over the domain, making the Riemann-Liouville and Caputo forms mathematically equivalent. While physically interpreted as a "hypothesis of quiescent past," analytically this is a direct consequence of the topological structure of the functional space.
\end{remark}

\begin{remark}[Unification of Fractional Calculi] \label{rem:unification}
A major advantage of the proposed framework is its ability to encompass a vast array of existing fractional operators on the real line as particular cases. By carefully selecting the memory kernel $k(z)$, the scale function $\psi(t)$, and the weight $\omega(t)$, the operator $\mathfrak{D}_{\psi, \omega}^{(k)}$ reduces to the operators listed in Table \ref{tab:special_cases}.
\end{remark}

\begin{table}[H]
\centering
\renewcommand{\arraystretch}{1.5}
\setlength{\tabcolsep}{8pt}
\caption{Classification of special cases of the Weighted Weyl-Sonine Operator. The first block lists kernels from the \textbf{Diffusive Regime} (Complete Bernstein Functions, monotone dynamics), while the last entry illustrates the \textbf{Oscillatory/Inertial Regime} (non-Bernstein, wave-like dynamics).}
\label{tab:special_cases}
\begin{tabular}{@{}llll@{}}
\toprule
\textbf{Operator Name} & \textbf{Kernel} $k(z)$ & \textbf{Scale} $\psi(t)$ & \textbf{Weight} $\omega(t)$ \\ \midrule
\multicolumn{4}{l}{\textit{\textbf{I. Diffusive Regime (Monotone / Bernstein)}}} \\
Standard Weyl & $\frac{z^{-\alpha}}{\Gamma(1-\alpha)}$ & $t$ & $1$ \\
Tempered Weyl & $\frac{z^{-\alpha}}{\Gamma(1-\alpha)}$ & $t$ & $e^{\lambda t}$ \\
Hadamard-type & $\frac{z^{-\alpha}}{\Gamma(1-\alpha)}$ & $\ln t$ & $1$ \\
Caputo-Fabrizio & $\exp\left(-\frac{\alpha}{1-\alpha}z\right)$ & $t$ & $1$ \\
Atangana-Baleanu & $E_\alpha\left(-\frac{\alpha}{1-\alpha}z^\alpha\right)$ & $t$ & $1$ \\ 
Distributed Order & $\int_0^1 b(\alpha) \frac{z^{-\alpha}}{\Gamma(1-\alpha)} d\alpha$ & $t$ & $1$ \\
\midrule
\multicolumn{4}{l}{\textit{\textbf{II. Oscillatory Regime (Inertial / Non-Bernstein)}}} \\
\textbf{Bessel-Klein-Gordon} & $J_0(2\sqrt{\lambda z})$ & $t$ & $1$ \\
\bottomrule
\end{tabular}
\end{table}

Table \ref{tab:special_cases} illustrates the unifying capability of the proposed framework.
While the first columns encompass geometric variations (time-scales and kernel shapes), the inclusion of the weight $\omega(t)$ reveals a deeper algebraic connection between weighted spaces and tempered operators.
In particular, the case designated as `Tempered Weyl' ($\omega(t) = e^{\lambda t}$) is often conceptually distinguished from `Tempered Memory' models (where the kernel itself decays as $e^{-\lambda t}k(t)$).
However, our formalism reveals that these two approaches are \textbf{algebraically equivalent}.

\begin{remark}[Equivalence of Space Tempering and Memory Tempering]
\label{rem:tempering_equivalence}
It is a non-trivial fact that weighting the function space is isomorphic to tempering the memory kernel. Consider the weighted integral $\mathfrak{I}_{\psi, \omega}^{(k)}$ with standard scale $\psi(t)=t$ and exponential weight $\omega(t) = e^{\lambda t}$. A direct computation shows:
\begin{equation}
    \mathfrak{I}_{t, e^{\lambda t}}^{(k)} u(t) 
    = e^{-\lambda t} \int_{-\infty}^t k(t-\tau) e^{\lambda \tau} u(\tau) \, d\tau 
    = \int_{-\infty}^t \left[ e^{-\lambda(t-\tau)} k(t-\tau) \right] u(\tau) \, d\tau.
\end{equation}
The term in brackets identifies a new effective kernel $k_{\lambda}(z) = e^{-\lambda z}k(z)$. Thus, the weighted operator on standard $L^2$ functions is identical to the unweighted convolution operator with a tempered kernel.
This confirms that our framework naturally generates the entire class of tempered fractional operators (such as the Tempered RL or Caputo derivatives) simply by selecting the appropriate weight $\omega$, without needing to redefine the kernel class ad-hoc.
\end{remark}
% ----------------------------------------------
% =================================================================================
%-----------------------

\subsection{Spectral Analysis}

To establish a direct connection with the classical theory of pseudo-differential operators, we characterize the Generalized Sonine pair through its spectral signature.

\begin{remark}[\textbf{The Spectral Advantage}]
    Luchko notes in \cite{Luchko2021} that an analytical description of the entire set of Sonine kernels remains an open problem. In the time domain, explicit closed-form formulas for the pairs $(k, \kappa)$ are often intractable or involve special functions that are computationally expensive to evaluate (e.g., H-functions).
    
    Our spectral framework bypasses this hurdle. By characterizing the operator via its symbol $\Phi(s)$ rather than its temporal kernel, we can solve fractional differential equations analytically even when the explicit form of the memory kernel $\kappa(t)$ is unknown, provided its Laplace symbol is well-defined. This is particularly useful in data-driven modeling, where $\Phi(s)$ is often fitted directly from rheological spectra.
\end{remark}

\begin{definition}\label{def:spectral_symbol}[The Spectral Symbol of Order $n$]
Let $n \in \mathbb{N}$ and let $(k, \kappa)$ be an admissible Sonine-Weyl pair of order $n$. We define the \textit{Spectral Symbol} $\Phi(s)$ (for $\text{Re}(s)>0$) as the Laplace transform of the associated derivative kernel $\kappa$:
\begin{equation}
\Phi(s) := \mathcal{L}\{\kappa\}(s).
\end{equation}
Recalling from Definition \ref{def:tempered_pair} that $(k*\kappa)(t) = t_+^{n-1}/(n-1)!$, the Laplace transform of the convolution implies the fundamental algebraic constraint:
\begin{equation}
    \mathcal{L}\{k\}(s) \cdot \Phi(s) = \frac{1}{s^n}.
\end{equation}
Consequently, the integral kernel $k$ is fully determined in the spectral domain by the inverse symbol:
\begin{equation}
\mathcal{L}\{k\}(s) = \frac{1}{s^n \Phi(s)}.
\label{eq:kernel_symbol_relation}
\end{equation}
\end{definition}

\begin{proposition}[Existence and Analytic Structure of Admissible Distributional Pairs]
\label{prop:sonine_existence}
Let $\Phi(s)$ be a function analytic in the right half-plane. A unique pair of tempered distributions $(k, \kappa) \in \mathcal{S}'_+(\mathbb{R}) \times \mathcal{S}'_+(\mathbb{R})$\footnote{The notation $\mathcal{S}'_+(\mathbb{R})$ denotes the subspace of tempered distributions supported on $[0, \infty)$. This condition imposes physical causality (the memory kernel vanishes for $t<0$) and ensures the existence of the Laplace transform $\Phi(s)$.} satisfying the generalized Sonine relation exists if and only if they are determined by the inverse Laplace transforms of \eqref{eq:k_construction}:
\begin{equation} \label{eq:k_construction}
    \mathcal{L}\{k\}(s) = \frac{1}{s^n \Phi(s)}, \quad \mathcal{L}\{\kappa\}(s) = \Phi(s).
\end{equation}
Moreover, regarding the physical admissibility within the distributional framework, we distinguish two levels of regularity based on the criteria established by Kochubei \cite{Kochubei2011} and Hanyga \cite{Hanyga2020}:

\begin{enumerate}
    \item \textbf{Singularity Condition (General Non-Locality):} To generate a valid non-local derivative, the distribution $\kappa$ must not be regular at the origin. This is satisfied if $\Phi(s)$ behaves asymptotically as a \textbf{Stieltjes function} (e.g., $\Phi(s) \sim s^{\alpha-n}$), implying that $\kappa$ is a singular distribution.
    
    \item \textbf{Thermodynamic Consistency (Diffusive Subclass):} Specifically for \textbf{dissipative/diffusive systems}, we require that the effective operator symbol $\Psi(s) = s^n \Phi(s)$ belongs to the class of \textbf{Complete Bernstein Functions} (CBF). This guarantees that the corresponding operator generates a positivity-preserving semi-group (see Remark \ref{rem:diffusive_vs_oscillatory}).
    
    \item \textbf{Conjugate Regularity:} If condition (2) holds (diffusive case with $n=1$), the conjugate kernel $k$ is also a Stieltjes distribution (typically a Locally Integrable Completely Monotone function), satisfying the distributional Sonine condition $k * \kappa = \delta$ within the class of singular measures.
\end{enumerate}
\end{proposition}

\begin{proof}
    \textit{Algebraic Existence in $\mathcal{S}'$.}
    The generalized condition $(k * \kappa)(t) = t_+^{n-1}/(n-1)!$ corresponds in the Laplace domain to the algebraic equation $\hat{k}(s)\hat{\kappa}(s) = s^{-n}$. Since the Laplace transform is an isomorphism on the space of supported distributions, identifying $\hat{\kappa}(s) = \Phi(s)$ uniquely determines $\hat{k}(s) = (s^n \Phi(s))^{-1}$ as a tempered distribution.

    \textit{Analytic Properties (Diffusive Case).}
    Assume strictly that the operator governs a diffusive process, so $\Psi(s) = s^n \Phi(s)$ is a Complete Bernstein Function (CBF). By the properties of CBFs, its reciprocal $1/\Psi(s)$ is a Stieltjes function. 
    The inverse Laplace transform of a Stieltjes function is known to be a Completely Monotone (CM) function (or measure). Thus, $k(t) = \mathcal{L}^{-1}\{1/\Psi(s)\}$ is a non-negative distribution with CM density.
    By duality, if $\Phi(s)$ is Stieltjes, then $\kappa$ is also a distribution with CM density (singular at the origin).
    This structural duality ensures that the resulting Weyl-Sonine operators preserve the cone of non-negative functions in $\mathcal{S}_{\psi,\omega}$, a necessary condition for the well-posedness of diffusion models. For oscillatory cases (where CBF conditions are relaxed), existence is guaranteed by the algebraic part, but positivity is not preserved.
\end{proof}

\begin{remark}[Comparison with Classical Frameworks: Riemann-Liouville vs. Weyl]
    It is important to contrast our construction with the classical theory developed by Luchko \cite{Luchko2021}. The classical approach primarily addresses \textbf{Riemann-Liouville type} operators defined on finite intervals $[0, T]$ acting on spaces of continuous functions (e.g., $C_{-1}$). In that setting, kernel pairs are strictly categorized by their integrability at the origin (classes $\mathcal{K}$ vs. $\mathcal{S}_{-1}$).
    
    In contrast, our work generalizes these concepts to \textbf{Weyl-type} operators defined on the entire real line $\mathbb{R}$. This extension is non-trivial as it requires controlling the asymptotic behavior of functions at infinity. By constructing the Weighted Weyl-Sonine operators over the distribution space $\mathcal{S}'_{\psi,\omega}$, we achieve two advancements:
    \begin{enumerate}
        \item We bypass the ad-hoc integrability conditions at the origin, as $\mathcal{S}'$ naturally accommodates hypersingular functionals.
        \item We extend the algebraic structure of the Sonine pairs to the full timeline, allowing for the spectral analysis of "aging" systems with infinite memory history, a feature not captured by the finite-start Riemann-Liouville definitions.
    \end{enumerate}
    Consequently, our Proposition \ref{prop:sonine_existence} adapts the physical admissibility criteria of Kochubei (Bernstein/Stieltjes) to the broader context of Weyl operators on tempered distributions.
\end{remark}

%---------------

We now proceed to the main spectral mapping theorem. It is worth noting that the classical fractional calculus algebra—characterized by the power law $s^{\alpha}$—is recovered here as a particular instance, now generalized by the spectral symbol $\Phi(s)$.

\begin{theorem}[Generalized Spectral Mapping Theorem]
\label{thm:spectral_mapping}
Let $n \in \mathbb{N}$ and let $u \in \mathcal{S}_{\psi,\omega}(\mathbb{R})$. Let $(k, \kappa)$ be an admissible distributional Sonine pair with associated Laplace symbol $\Phi(s) := \hat{\kappa}(s)$. The Weighted Weyl-Sonine operators satisfy the following spectral relations:

\begin{enumerate}
    \item \textbf{For the Integral:}
    \begin{equation}
        \mathcal{F}_{\psi,\omega}\{\mathfrak{I}_{\psi,\omega}^{(k)}u\}(\xi) = \frac{1}{(i\xi)^n \Phi(i\xi)} \cdot \mathcal{F}_{\psi,\omega}\{u\}(\xi).
    \end{equation}

    \item \textbf{For the Derivative:} Considering the operator $\mathfrak{D}_{\psi,\omega}^{(\kappa)}$ from Definition \ref{def:derivada}, the spectral mapping holds:
    \begin{equation}
        \mathcal{F}_{\psi,\omega}[\mathfrak{D}_{\psi,\omega}^{(\kappa)}u](\xi) = (i\xi)^n \Phi(i\xi) \cdot \mathcal{F}_{\psi,\omega}[u](\xi).
    \end{equation}
\end{enumerate}
\end{theorem}

\begin{proof}
We rely on the structural factorization of the Weighted Fourier Transform established in Section 3.2: $\mathcal{F}_{\psi,\omega} = \mathcal{F} \circ Q_{\psi}^{-1} \circ M_{\omega}$, which acts on the auxiliary function $v = Q_{\psi}^{-1} (M_{\omega} u) \in \mathcal{S}(\mathbb{R})$.

\textbf{Case 1 (Integral):} 
Recall from Definition \ref{def:weighted_integral} that the integral operator admits the conjugation structure:
$$
\mathfrak{I}_{\psi,\omega}^{(k)} = M_{\omega}^{-1} Q_{\psi} (k *) Q_{\psi}^{-1} M_{\omega}.
$$
Applying the Weighted Fourier Transform to the operator $\mathfrak{I}_{\psi,\omega}^{(k)} u$ leads to a direct cancellation of the structural maps:
\begin{align*}
\mathcal{F}_{\psi,\omega} \left( \mathfrak{I}_{\psi,\omega}^{(k)} u \right) &= (\mathcal{F} \circ Q_{\psi}^{-1} \circ M_{\omega}) \circ (M_{\omega}^{-1} Q_{\psi} (k *) Q_{\psi}^{-1} M_{\omega} u) \\
&= \mathcal{F} \left( k * (Q_{\psi}^{-1} M_{\omega} u) \right) = \mathcal{F} \{ k * v \}.
\end{align*}
Since $v \in \mathcal{S}(\mathbb{R})$ and $k \in \mathcal{S}'_+$, we apply the Convolution Theorem for tempered distributions. Using the Sonine condition in the frequency domain, $\hat{k}(i\xi)\hat{\kappa}(i\xi) = (i\xi)^{-n}$, we relate the symbol of $k$ to $\Phi$:
$$
\mathcal{F}\{k * v\}(\xi) = \hat{k}(i\xi) \cdot \mathcal{F}\{v\}(\xi) = \frac{1}{(i\xi)^n \hat{\kappa}(i\xi)} \mathcal{F}\{v\}(\xi).
$$
Substituting $\Phi(i\xi) = \hat{\kappa}(i\xi)$ and identifying $\mathcal{F}\{v\} = \mathcal{F}_{\psi,\omega}\{u\}$, the first result follows.

\textbf{Case 2 (Derivative):} 
For the derivative operator of order $n$, we utilize the structural form given in Eq. \eqref{eq:sonine_abstract}:
$$
\mathfrak{D}_{\psi,\omega}^{(\kappa)} = M_{\omega}^{-1} Q_{\psi} \left( \frac{d^n}{dy^n} \circ (\kappa *) \right) Q_{\psi}^{-1} M_{\omega}.
$$
Applying $\mathcal{F}_{\psi,\omega}$ again cancels the outer geometrical factors, yielding the Fourier transform of a convoluted derivative in the Schwartz space:
$$
\mathcal{F}_{\psi,\omega}[\mathfrak{D}_{\psi,\omega}^{(\kappa)}u](\xi) = \mathcal{F} \left[ \frac{d^n}{dy^n} (\kappa * v) \right](\xi).
$$
Using the derivative property of the classical Fourier transform iteratively, $\mathcal{F}\{f^{(n)}\}(\xi) = (i\xi)^n \mathcal{F}\{f\}(\xi)$, we obtain:
$$
\mathcal{F} \left[ \frac{d^n}{dy^n} (\kappa * v) \right](\xi) = (i\xi)^n \cdot \mathcal{F}\{\kappa * v\}(\xi) = (i\xi)^n \Phi(i\xi) \mathcal{F}\{v\}(\xi).
$$
Recovering $\mathcal{F}_{\psi,\omega}[u](\xi)$ from $\mathcal{F}\{v\}(\xi)$ completes the proof.
\end{proof}

\begin{remark}[Natural Domains and Future Outlook]
\label{rem:sobolev_spaces}
The Generalized Spectral Mapping Theorem implies that the natural energy spaces for these operators are the \textit{Weighted Sonine-Sobolev spaces}, defined via the magnitude of the \textbf{full operator symbol} $\Psi(i\xi) = (i\xi)^n \Phi(i\xi)$:
$$
\mathcal{H}_{\psi,\omega}^{\Phi}(\mathbb{R}) := \left\{ u \in L^2_{\psi,\omega}(\mathbb{R}) : \int_{\mathbb{R}} \left(1 + |(i\xi)^n \Phi(i\xi)|^2 \right) |\mathcal{F}_{\psi,\omega}u(\xi)|^2 \, d\xi < \infty \right\}.
$$
While the construction of these spaces is naturally suggested by our spectral framework, a rigorous treatment—including density results, embedding theorems, and trace properties involving the weights $(\psi, \omega)$—lies outside the scope of the present work. These analytic aspects constitute a substantial topic that will be addressed in a forthcoming study dedicated to the regularity of solutions.
\end{remark}

%----------------------------------------

\subsection{The Generalized Marchaud Representation}
\label{sec:marchaud}

To fully align our framework with the analytical philosophy of non-local operators (see Stinga \cite{Stinga} and Samko \cite{Samko1993}), we establish a connection with hypersingular integral representations. While the Sonine form (convolution) is convenient for algebraic manipulation, the Marchaud form (finite differences) highlights the non-local nature of the operator as an accumulation of `jumps' governed by a L\'evy measure.

We show that our Weighted Weyl-Sonine derivative admits such a representation, provided the spectral symbol $\Phi(s)$ corresponds to a Bernstein function associated with differentiation orders strictly between 0 and 1.

\begin{definition}[Associated L\'evy Density]
\label{def:levy}
Let $\Phi(s)$ be the spectral symbol defined in Definition \ref{def:spectral_symbol}. Assume $\Phi$ admits the L\'evy-Khintchine representation for a drift-free subordinator:
\begin{equation} \label{eq:levy_khintchine}
    \Phi(s) = \int_0^\infty (1 - e^{-sz}) \, d\mu(z),
\end{equation}
where $\mu$ is a L\'evy measure on $(0, \infty)$ satisfying $\int_0^\infty \min(1, z) \, d\mu(z) < \infty$. Furthermore, we assume $\mu$ is absolutely continuous with respect to the Lebesgue measure, admitting a L\'evy density $g(z)$ such that $d\mu(z) = g(z)dz$.
\end{definition}

\begin{theorem}[Weighted Marchaud Formula]
\label{thm:marchaud_equivalence}
Under the assumptions of Definition \ref{def:levy}, the Weighted Weyl-Sonine derivative $\mathfrak{D}_{\psi,\omega}^{(\kappa)}$ admits the following generalized Marchaud representation for any $u \in \mathcal{S}_{\psi,\omega}$:
\begin{equation} \label{eq:marchaud_formula}
    \mathfrak{D}_{\psi,\omega}^{(\kappa)} u(t) = \frac{1}{\omega(t)} \int_{-\infty}^t \Big( (\omega u)(t) - (\omega u)(\tau) \Big) g\Big(\psi(t) - \psi(\tau)\Big) \psi'(\tau) \, d\tau.
\end{equation}
\end{theorem}

\begin{proof}
We employ the uniqueness of the spectral representation established in Theorem \ref{thm:spectral_mapping}. Our goal is to show that the right-hand side of \eqref{eq:marchaud_formula}, under the conjugation map $\mathcal{F}_{\psi,\omega}$, yields the correct spectral symbol $\Phi$.

Let $u \in \mathcal{S}_{\psi,\omega}$ and define the auxiliary function $v = Q_{\psi}^{-1}M_{\omega}u$ on the rectified timeline $y = \psi(t)$. Recall that $(M_\omega u)(t) = v(\psi(t))$. Substituting this into the integral term of \eqref{eq:marchaud_formula}, the difference term becomes:
$$
(\omega u)(t) - (\omega u)(\tau) = v(\psi(t)) - v(\psi(\tau)).
$$
Let $I(t)$ denote the integral in \eqref{eq:marchaud_formula}. We analyze the operator on the rectified domain by applying the conjugation steps and evaluating at $t = \psi^{-1}(y)$:
$$
(Q_{\psi}^{-1}M_{\omega} I)(y) = \int_{-\infty}^{\psi^{-1}(y)} \Big( v(y) - v(\psi(\tau)) \Big) g\Big(y - \psi(\tau)\Big) \psi'(\tau) \, d\tau.
$$
We perform the change of variables $z = y - \psi(\tau)$. This implies $dz = -\psi'(\tau)d\tau$. The integration limits transform as follows: as $\tau \to -\infty$, $z \to \infty$; and as $\tau \to \psi^{-1}(y)$, $z \to 0$. Absorbing the minus sign into the limits, we obtain:
\begin{equation*}
(Q_{\psi}^{-1}M_{\omega} I)(y) = \int_{0}^{\infty} \Big( v(y) - v(y-z) \Big) g(z) \, dz.
\end{equation*}
This represents a convolution of differences. We take the classical Fourier transform $\mathcal{F}$ with respect to $y$. Using the linearity of the integral and Fubini's theorem (justified by the convergence of the L\'evy measure against bounded differences), we interchange the Fourier integral with the $dz$ integral:
\begin{align*}
\mathcal{F}\left\{ \int_{0}^{\infty} (v(\cdot) - v(\cdot-z)) g(z) \, dz \right\}(\xi) &= \int_{0}^{\infty} \mathcal{F}\Big\{ v(y) - v(y-z) \Big\}(\xi) \, g(z) \, dz \\
&= \int_{0}^{\infty} \Big( \hat{v}(\xi) - e^{-i\xi z}\hat{v}(\xi) \Big) g(z) \, dz \\
&= \hat{v}(\xi) \left( \int_{0}^{\infty} (1 - e^{-i\xi z}) g(z) \, dz \right).
\end{align*}
By the L\'evy-Khintchine formula \eqref{eq:levy_khintchine} (with $s=i\xi$), the integral in parentheses is exactly $\Phi(i\xi)$. Thus, the spectral symbol matches Definition \ref{def:spectral_symbol}, completing the proof.
\end{proof}

\begin{remark}
This result unifies the algebraic approach with the analytical perspective. For the classical fractional case where $\Phi(s) = s^\alpha$ (with $\alpha \in (0,1)$), the associated L\'evy density is $g(z) = \frac{\alpha}{\Gamma(1-\alpha)} z^{-(1+\alpha)}$. Substituting this into \eqref{eq:marchaud_formula} recovers the weighted extension of the classical Marchaud fractional derivative.
\end{remark}

\begin{remark}[Abstract Operator Unification]
Before proceeding to the applications, it is worth noting that the structure presented here is an instance of a broader algebraic construction. If $\mathcal{T}: \mathcal{H} \to L^2(\mathbb{R})$ is any generic invertible linear operator, one can define an abstract Weyl-Sonine operator as $\mathfrak{D}_{\mathcal{T}} := \mathcal{T}^{-1} \mathbb{D}^{(\kappa)} \mathcal{T}$.
In this scenario, our "Weighted Fourier Transform" generalizes naturally to a \textit{Conjugated Fourier Transform} defined by $\mathcal{F}_{\mathcal{T}} := \mathcal{F} \circ \mathcal{T}$.
Remarkably, the Spectral Mapping Theorem proved earlier holds invariantly:
\begin{equation}
    \mathcal{F}_{\mathcal{T}} (\mathfrak{D}_{\mathcal{T}} u)(\xi) = \Phi(i\xi) \mathcal{F}_{\mathcal{T}} u(\xi).
\end{equation}
Our choice of the specific weighted composition operator $\mathcal{T}u(t) = \omega(t)u(\psi(t))$ is motivated by its direct physical interpretation in modeling aging media (time-stretching) and tempered dynamics, providing a concrete realization of this abstract spectral equivalence.
\end{remark}
%-------------------------------------------------------------------------------------------------------
\section{Applications to Anomalous Relaxation, Diffusion, and Wave Dynamics}
\label{sec:applications}

In this final section, we demonstrate the versatility of the Weighted Weyl-Sonine framework by solving evolution problems that arise in complex media. We address three distinct physical regimes: generalized relaxation (diffusive), distributed-order evolution (ultra-slow transitions), and inertial propagation (oscillatory dynamics).

\begin{remark}
\label{rem:evolution_scope}
The theoretical framework developed in Section \ref{sec:weyl_sonine_framework} supports operators of arbitrary order $n$. In this section, we primarily analyze the linear evolution equation $\mathfrak{D}_{\psi, \omega}^{(\kappa)} u + \lambda u = f$. This formulation is general enough to unify the description of purely dissipative systems (relaxation/diffusion) and systems with inertial energy exchange (waves), depending on the spectral choice of the kernel $\kappa$.
\end{remark}

\subsection{Spectral Solvability of Generalized Evolution Models}

We consider the fundamental evolution equation governing the dynamics of a field $u \in \mathcal{S}_{\psi, \omega}$ (e.g., stress in a viscoelastic material, density in a porous medium, or amplitude of a wave packet) under the action of the generalized derivative.

\textbf{Problem Statement:} Find a unique solution $u(t)$ to the Cauchy-type problem on the real line:
\begin{equation} \label{eq:evolution_prob}
    \mathfrak{D}_{\psi, \omega}^{(\kappa)} u(t) + \lambda u(t) = f(t), \quad t \in \mathbb{R},
\end{equation}
where $\lambda > 0$ is a structural parameter and $f \in L^2_{\psi, \omega}(\mathbb{R})$ is a source term.
The physical nature of the solution is entirely determined by the spectral class of the kernel $\kappa$:
\begin{itemize}
    \item \textbf{Diffusive Case:} If $\kappa$ is a Bernstein function (Table \ref{tab:special_cases}, Block I), the system exhibits monotonic relaxation.
    \item \textbf{Oscillatory Case:} If $\kappa$ allows for inertial effects (e.g., Bessel kernel, Table \ref{tab:special_cases}, Block II), the system supports wave propagation.
\end{itemize}

Using the Generalized Spectral Mapping Theorem (Theorem \ref{thm:spectral_mapping}), we can characterize the solvability purely in terms of the spectral symbol $\Phi$.

\begin{theorem}[Well-posedness and Spectral Resolution]
\label{thm:well_posedness}
Let $\Phi(s) = \mathcal{L}\{\kappa\}(s)$ be the spectral symbol of the derivative kernel. Assume the following \textbf{Ellipticity Condition} holds on the imaginary axis:
\begin{equation} \label{eq:spectral_condition}
    \inf_{\xi \in \mathbb{R}} \left| i\xi \Phi(i\xi) + \lambda \right| > 0.
\end{equation}
Then, the operator $(\mathfrak{D}_{\psi, \omega}^{(\kappa)} + \lambda I)$ is a linear isomorphism from the domain $\mathcal{H}_{\psi,\omega}^{\Phi}$ onto $L^2_{\psi, \omega}(\mathbb{R})$. The unique solution is given by the weighted convolution:
\begin{equation} \label{eq:exact_sol}
    u(t) = \frac{1}{\omega(t)} \int_{-\infty}^{t} \mathcal{G}_{\lambda}( \psi(t) - \psi(\tau) ) \, f(\tau) \, \omega(\tau) \, \psi'(\tau) \, d\tau,
\end{equation}
where the Green's function $\mathcal{G}_{\lambda}$ is determined by the inverse Fourier transform of the resonant symbol:
\begin{equation} \label{eq:greens_symbol}
    \widehat{\mathcal{G}_{\lambda}}(\xi) = \frac{1}{i\xi \Phi(i\xi) + \lambda}.
\end{equation}
\end{theorem}

\begin{proof}
Applying the Weighted Fourier Transform $\mathcal{F}_{\psi, \omega}$ to Eq. \eqref{eq:evolution_prob} and utilizing Theorem \ref{thm:spectral_mapping} (with $n=1$), the differential equation converts into an algebraic one in the frequency domain. Recall that for $n=1$, the symbol of the operator is $i\xi \Phi(i\xi)$:
\begin{equation*}
    (i\xi \Phi(i\xi) + \lambda) \mathcal{F}_{\psi, \omega}\{u\}(\xi) = \mathcal{F}_{\psi, \omega}\{f\}(\xi).
\end{equation*}
Under condition \eqref{eq:spectral_condition}, the multiplier $m(\xi) = (i\xi \Phi(i\xi) + \lambda)^{-1}$ belongs to $L^\infty(\mathbb{R})$. Since the Weighted Fourier Transform is a unitary isomorphism between $L^2_{\psi, \omega}$ and $L^2$, the solution is uniquely determined in the spectral domain. The explicit time-domain representation \eqref{eq:exact_sol} follows directly from the inverse structure $\mathcal{F}_{\psi, \omega}^{-1} = M_{\omega}^{-1}\circ Q_{\psi}\circ \mathcal{F}^{-1}$.
\end{proof}

% -------------------------------------------------------------------------

\subsection{Application: Weighted Distributed Order and Ultra-Slow Transitions}

A significant advantage of the Weyl-Sonine framework is its ability to encapsulate operators of \textit{distributed order}. These operators model complex multi-scale memory effects and "ultra-slow" diffusion processes where the mean squared displacement grows logarithmically rather than as a power law (a behavior distinct from the smooth crossover to normal diffusion depicted in Figure \ref{fig:msd}). % <--- AQUI AGREGUÉ LA REFERENCIA COMPARATIVA

% --- FIGURA 2: MSD DYNAMICS ---
% Ubicación: En la sección de Distributed Order / Ultra-slow transitions
\begin{figure}[htbp]
    \centering
    \includegraphics[width=0.75\textwidth]{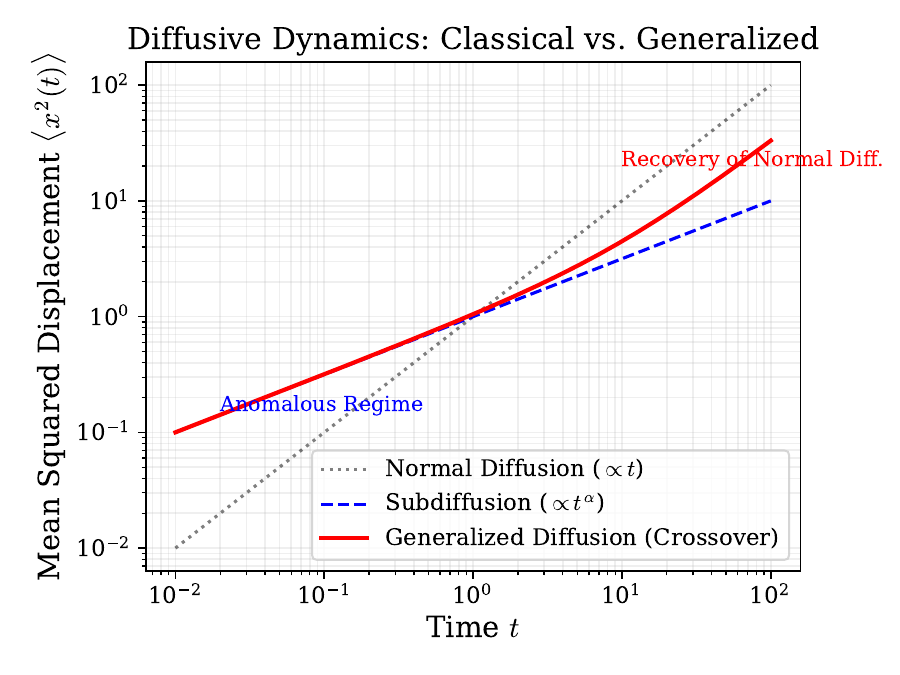}
    \caption{\textbf{Diffusive Regimes.} Mean Squared Displacement (MSD) for different memory structures. The plot illustrates the transition from subdiffusion ($\alpha < 1$) to normal diffusion ($\alpha = 1$) typical of tempered systems. Note that the \textit{retarded aging} process discussed in the text follows an ultra-slow logarithmic trajectory (flatter than the curves shown here).}
    \label{fig:msd}
\end{figure}

Our framework allows for a natural generalization of distributed order operators to time-varying domains via the weights $(\psi, \omega)$.

\begin{definition}[Weighted Distributed Order Operator]
\label{def:dist_order}
Let $b: [0,1] \to [0, \infty)$ be a probability density function representing the distribution of differentiation orders. We define the Weighted Distributed Order Operator, $\mathbb{D}_{\psi, \omega}^{b}$, as the integral superposition:
\begin{equation}
    \mathbb{D}_{\psi, \omega}^{b} u(t) := \int_0^1 b(\alpha) \, \mathfrak{D}_{\psi, \omega}^\alpha u(t) \, d\alpha.
\end{equation}
This operator belongs to the Sonine class with order $n=1$. Its associated memory kernel is $\kappa_{\text{eff}}(t) = \int_0^1 \frac{b(\alpha)}{\Gamma(1-\alpha)} t^{-\alpha} \, d\alpha$. To maintain consistency with Theorem \ref{thm:well_posedness}, we define the \textbf{Spectral Symbol} $\Phi_b(s)$ as the transform of this kernel:
\begin{equation} \label{eq:dist_symbol_def}
    \Phi_b(s) = \mathcal{L}\{\kappa_{\text{eff}}\}(s) = \int_0^1 b(\alpha) s^{\alpha-1} \, d\alpha.
\end{equation}
\end{definition}

Note that while the intuitive symbol of the operator is $\int b(\alpha) s^\alpha d\alpha$, our structural parameter $\Phi_b(s)$ includes the singular factor $s^{-1}$ inherent to the memory kernel definition. Since $s^{\alpha-1}$ is the Stieltjes transform of a measure for $\alpha \in (0,1)$, $\Phi_b(s)$ preserves the necessary analytic properties.

\subsubsection{Example: Retarded Aging via Logarithmic Diffusion}
Consider the canonical case of a uniform distribution of orders, $b(\alpha) = 1$. We compute the spectral symbol according to \eqref{eq:dist_symbol_def}:
\begin{equation}
    \Phi_{\text{uni}}(i\xi) = \int_0^1 (i\xi)^{\alpha-1} \, d\alpha = \frac{1}{i\xi} \int_0^1 (i\xi)^\alpha \, d\alpha = \frac{i\xi - 1}{i\xi \ln(i\xi)}.
\end{equation}
Now we apply Theorem \ref{thm:well_posedness}. The characteristic multiplier for the Green's function is determined by substituting $\Phi_{\text{uni}}$ into Eq. \eqref{eq:greens_symbol}:
$$
i\xi \Phi_{\text{uni}}(i\xi) + \lambda = i\xi \left( \frac{i\xi - 1}{i\xi \ln(i\xi)} \right) + \lambda = \frac{i\xi - 1}{\ln(i\xi)} + \lambda.
$$
Thus, the Green's function is determined by the inverse Fourier transform of:
\begin{equation}
    \widehat{\mathcal{G}_{\lambda}}(\xi) = \frac{\ln(i\xi)}{i\xi - 1 + \lambda \ln(i\xi)}.
\end{equation}
In classical invariant media, this symbol corresponds to ultra-slow logarithmic diffusion (see Chechkin et al. \cite{Chechkin2002}). In our weighted framework, this describes a process of \textit{retarded aging}: the system relaxes logarithmically with respect to the subjective internal time $\psi(t)$, modulated by the changing density $\omega(t)$. This model is particularly relevant for biological growth or materials with evolving microstructure where the "clock" slows down over time.

\subsection{Application: Inertial Dynamics and Wave Propagation in Expanding Media}
\label{sec:inertial_dynamics}

Finally, we address the \textit{Oscillatory Regime} identified in Remark \ref{rem:diffusive_vs_oscillatory}. Unlike the relaxation and diffusion processes discussed above, inertial dynamics (such as wave propagation) require memory kernels that allow for energy exchange rather than pure dissipation.

Consider the generalized telegrapher's equation in a time-varying medium. In our framework, this is modeled by selecting the \textbf{Bessel kernel} from Table \ref{tab:special_cases} (entry II):
\begin{equation}
    \kappa_{\text{wave}}(t) = J_0(2\sqrt{\gamma t}), \quad \gamma > 0.
\end{equation}
This kernel is \textbf{not} a Bernstein function; it is signed and oscillatory. Its spectral symbol is given by the essential singularity:
\begin{equation}
    \Phi_{\text{wave}}(s) = \mathcal{L}\{J_0(2\sqrt{\gamma t})\}(s) = \frac{1}{s} e^{-\gamma/s}.
\end{equation}

\textbf{Spectral Analysis of Oscillations:}
Using the spectral resolution provided by Theorem \ref{thm:well_posedness}, the dynamic behavior is governed by the poles of the Green's function $\mathcal{G}_\lambda$. The characteristic equation (setting the denominator of Eq. \eqref{eq:greens_symbol} to zero) becomes:
\begin{equation}
    i\xi \Phi_{\text{wave}}(i\xi) + \lambda = i\xi \left( \frac{1}{i\xi} e^{-\gamma/i\xi} \right) + \lambda = e^{i (\gamma/\xi)} + \lambda = 0.
\end{equation}
The dispersion relation $e^{i (\gamma/\xi)} + \lambda = 0$ implies that for modes on the unit spectral circle ($|\lambda|=1$), the wavenumbers $\xi$ are real. This creates a spectrum of \textbf{non-decaying propagating modes} (neutral stability), a hallmark of wave mechanics that contrasts sharply with the dissipative spectrum of parabolic operators.

\textbf{Physical Interpretation of the Weights:}
The solution $u(t)$ obtained via \eqref{eq:exact_sol} describes a wave packet. The structural functions play a geometric role:
\begin{itemize}
    \item The scale $\psi(t)$ acts as a \textit{local clock}. If $\psi(t) = t^\beta$ with $\beta > 1$ (expanding medium), the perceived frequency of the wave undergoes a redshift over time.
    \item The weight $\omega(t)$ modulates the amplitude. An exponential weight $\omega(t) = e^{\mu t}$ introduces a generalized friction or amplification term without altering the causality of the wave equation.
\end{itemize}
% --- FIGURA 3: INERTIAL DYNAMICS ---
% Ubicación: En la sección de Ondas, tras explicar el "reloj local"
\begin{figure}[H]
    \centering
    % Este es el PDF que generas con el script de Python
    \includegraphics[width=0.85\textwidth]{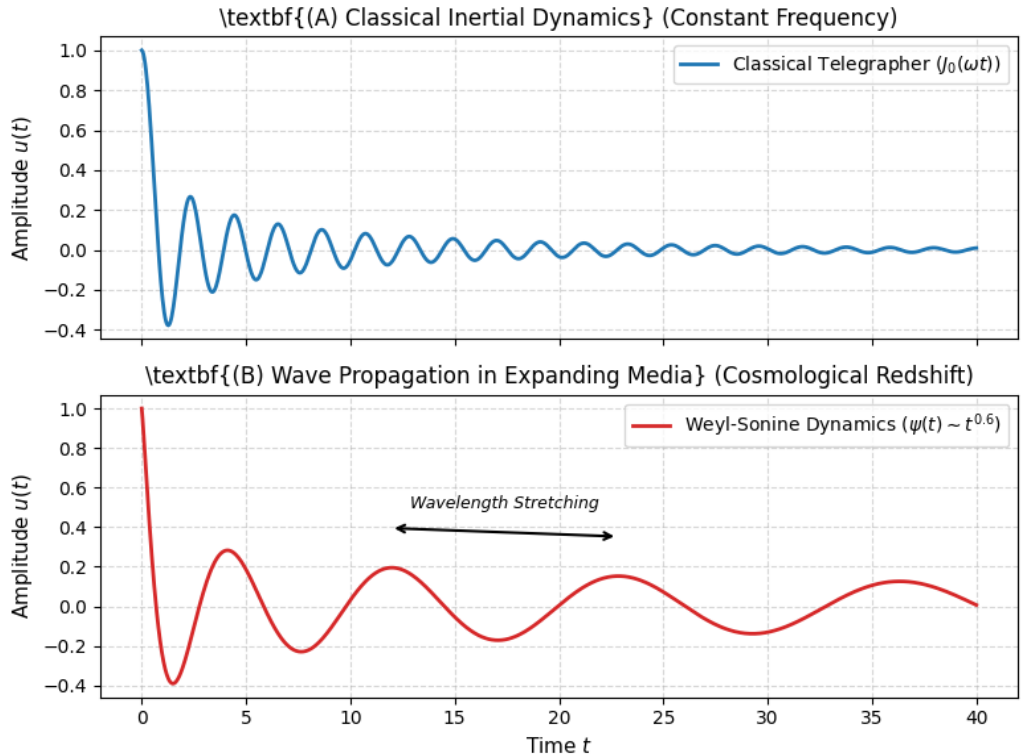}
    \caption{\textbf{Wave Propagation in Expanding Media.} 
    \textbf{(A)} Classical inertial dynamics (Telegrapher's equation) showing constant oscillation frequency. 
    \textbf{(B)} Weyl-Sonine evolution with a time-scale deformation $\psi(t) \sim t^{0.6}$. The wave packet exhibits a \textit{cosmological redshift}, where the wavelength stretches as the internal clock slows down relative to laboratory time.}
    \label{fig:redshift}
\end{figure}
This example confirms that the Weighted Weyl-Sonine framework unifies parabolic (diffusive) and hyperbolic (wave-like) phenomena under a single algebraic structure, distinguishing them solely by the spectral class of the memory kernel $\Phi$.

%\end{latex}
\section{Conclusions and Perspectives}
\label{sec:conclusions}

In this work, we have successfully constructed a unified mathematical framework that bridges the gap between General Fractional Calculus and the analysis of heterogeneous, aging media. By defining the Weighted Weyl-Sonine operators, we have extended the scope of non-local theory to the entire real line $\mathbb{R}$, allowing for memory kernels of arbitrary spectral nature—encompassing both the diffusive regime (governed by Complete Bernstein Functions) and the inertial regime (oscillatory kernels)—while simultaneously incorporating time-dependent material properties via the admissible structure pair $(\psi, \omega)$.

The core contribution of this paper is the establishment of the Generalized Spectral Mapping Theorem. We have demonstrated that, under the appropriate conjugation, the Weighted Fourier Transform acts as a unitary diagonalization map for these complex operators. This result is of significant practical value: it reduces the solution of variable-coefficient integro-differential equations to algebraic manipulations involving the spectral symbol $\Phi(i\xi)$.

The versatility of this approach was illustrated through the solvability of evolution equations across distinct physical landscapes. We derived explicit Green's functions for arbitrary admissible symbols, distinguishing between monotonic relaxation (where positivity is strictly preserved) and wave propagation (where oscillations arise from singular spectra). Additionally, we demonstrated that our framework naturally encompasses Weighted Distributed Order Operators, proving that phenomena of ultra-slow diffusion and retarded aging can be rigorously analyzed using the distributed spectral symbol $\Phi_b(\xi)$.

Furthermore, the derivation of the Marchaud-type representation (Theorem \ref{thm:marchaud_equivalence}) provides a crucial link between the algebraic operator definition and the analytical perspective of finite differences and L\'evy measures. This representation is theoretically significant as it relaxes the regularity requirements on the function space, and practically important as it opens the door to robust numerical implementations based on hypersingular quadratures.

Looking forward, several promising research avenues emerge from this study:

\begin{itemize}
    \item \textbf{Weighted Sonine-Sobolev Spaces:} As hinted in our spectral analysis, a rigorous treatment of the functional spaces $\mathcal{H}_{\psi,\omega}^{\Phi}$ is necessary. Future work will address embedding theorems and trace properties for these domains, providing the functional analytic bedrock for non-linear problems.
    
    \item \textbf{Inverse Problems:} The explicit relation between the spectral symbol $\Phi$ and the Green's function suggests the possibility of reconstructing the memory kernel or the aging weight $\omega(t)$ from spectral data. This is particularly relevant for identifying material properties in experimental rheology.
    
    \item \textbf{Numerical Schemes:} The Marchaud representation derived here lends itself naturally to the development of fast numerical schemes, avoiding the singularity issues often found in direct Riemann-Liouville discretizations of aging problems.
\end{itemize}

We believe this framework provides the necessary mathematical tools to model the next generation of complex systems where memory, heterogeneity, and aging coexist.
\clearpage
\appendix

\section{Well-posedness of Convolution Integrals}
\label{appendix:convergence}

In this appendix, we provide a rigorous justification for the convergence of the Weighted Weyl-Sonine integrals defined in Section 3, ensuring they are well-defined for kernels with both local singularities and tempered growth at infinity.

\begin{proposition}
Let $k \in L_{loc}^1(\mathbb{R})$ be a kernel satisfying the polynomial growth condition $|k(t)| \le C(1+|t|)^M$ for large $|t|$. Let $u \in \mathcal{S}_{\psi, \omega}$ be a test function in the Weighted Schwartz Space. Then, the convolution integral defining the operator converges absolutely for every $t \in \mathbb{R}$.
\end{proposition}

\begin{proof}
By the structural factorization theorem, the operator is unitarily equivalent to the convolution $(k * v)(y)$ on the rectified line, where $v \in \mathcal{S}(\mathbb{R})$ is the conjugated test function. To prove convergence, we split the integral into a local singularity region and a distal tail region:
$$
|(k * v)(y)| \le \int_{|z| \le 1} |k(z)||v(y-z)| \, dz + \int_{|z| > 1} |k(z)||v(y-z)| \, dz.
$$
\textbf{1. Local Singularities:}
Since $v \in \mathcal{S}(\mathbb{R})$, it is bounded ($v \in L^\infty$). For the local part, we rely on the local integrability of $k$:
$$
\int_{|z| \le 1} |k(z)||v(y-z)| \, dz \le \|v\|_\infty \int_{-1}^{1} |k(z)| \, dz < \infty.
$$
This ensures that weak singularities (e.g., $z^{-\alpha}$ with $0 < \alpha < 1$) do not prevent convergence.

\textbf{2. Growth at Infinity:}
For the tail $|z| > 1$, we exploit the rapid decay of the Schwartz function $v$. For any $N > 0$, there exists a constant $C_N$ such that $|v(y-z)| \le C_N (1+|y-z|)^{-N}$.
We invoke the classic Peetre's inequality, which states that for any $s \in \mathbb{R}$, $(1+|a+b|)^s \le (1+|a|)^s (1+|b|)^{|s|}$. Applying this with exponent $-N$, we have:
$$
(1+|y-z|)^{-N} \le (1+|y|)^N (1+|z|)^{-N}.
$$
Substituting this bound into the integral and using the polynomial growth of $k$:
$$
\int_{|z| > 1} |k(z)||v(y-z)| \, dz \le C \cdot C_N (1+|y|)^N \int_{|z| > 1} |z|^M (1+|z|)^{-N} \, dz.
$$
Choosing $N > M+1$ ensures that the integral on the right-hand side converges. Thus, the convolution is well-defined pointwise and inherits the polynomial growth of the kernel.
\end{proof}

\section{Extension to Weighted Distributions}
\label{appendix:distributional_fourier}

In this section, we justify the application of the Weighted Fourier Transform $\mathcal{F}_{\psi, \omega}$ to the singular kernels involved in our operator definitions. This extension relies on the topological duality established in our previous work regarding the space $\mathcal{S}_{\psi, \omega}$.

\subsection{The Weighted Dual Space}
We define the space of \textbf{Weighted Tempered Distributions}, denoted by $\mathcal{S}'_{\psi, \omega}$, as the topological dual of the Weighted Schwartz Space $\mathcal{S}_{\psi, \omega}$. A linear functional $T$ belongs to $\mathcal{S}'_{\psi, \omega}$ if and only if its conjugated version $\widetilde{T} = \mathcal{C}T$ defines a classical tempered distribution in $\mathcal{S}'(\mathbb{R})$, where $\mathcal{C} = Q_{\psi}^{-1}M_{\omega}$ is the canonical isomorphism.

\subsection{Definition via Parseval's Identity}
Since $\mathcal{F}_{\psi, \omega}$ is a unitary isomorphism on the test space $\mathcal{S}_{\psi, \omega}$, we can define its action on distributions via the standard duality method (Parseval's identity).

\begin{definition}[Distributional Weighted Fourier Transform]
Let $T \in \mathcal{S}'_{\psi, \omega}$ be a weighted tempered distribution. We define its Weighted Fourier Transform, denoted by $\widehat{T} = \mathcal{F}_{\psi, \omega}[T]$, as the unique distribution satisfying:
\begin{equation}
    \langle \widehat{T}, \varphi \rangle := \langle T, \mathcal{F}_{\psi, \omega}[\varphi] \rangle, \quad \forall \varphi \in \mathcal{S}_{\psi, \omega}.
\end{equation}
\end{definition}

\begin{remark}
This definition is consistent with the Plancherel isometry on $L^2_{\psi, \omega}$. It ensures that for any convolution kernel $\kappa$ (viewed as a distribution), the spectral mapping property holds rigorously:
$$
\mathcal{F}_{\psi, \omega} \left[ \mathfrak{D}_{\psi, \omega}^{(\kappa)} u \right] = \mathcal{F}_{\psi, \omega}[\kappa] \cdot \mathcal{F}_{\psi, \omega}[u]
$$
in the sense of distributions, where the product is understood as the multiplication of a distribution by the corresponding spectral symbol (multiplier function) or via the exchange formula.
\end{remark}
\section*{Declaration of competing interest}
The author declares that he has no known competing financial interests or personal relationships that could have appeared to influence the work reported in this paper.

%\section*{Acknowledgements}
%The author thanks the Universidad Nacional del Nordeste (UNNE) for the financial support under Grant [Insert Grant Number here if applicable].

\section*{Declaration of Generative AI and AI-assisted technologies in the manuscript preparation process}

During the preparation of this work the author(s) used Gemini (Google) in order to improve the readability and language quality of the manuscript, as well as to assist in structuring the discussion of the physical interpretations. After using this tool/service, the author(s) reviewed and edited the content as needed and take(s) full responsibility for the content of the published article.

% --- BIBLIOGRAFÍA ---

\end{document}